\documentclass[12pt]{amsart}

 \usepackage[pagewise]{lineno}
 
\usepackage[colorlinks=true,pagebackref,hyperindex,citecolor=blue,linkcolor=red]{hyperref}
\usepackage{amsmath}
\usepackage{amsfonts}
\usepackage{amssymb}
\usepackage{setspace}
\usepackage{moreverb}
\usepackage[latin1]{inputenc}
\usepackage[T1]{fontenc}
\usepackage{amssymb}
\usepackage{tikz}
\usepackage{color}
\usepackage[all]{xy}
\usepackage{xfrac}
\usepackage[top=1in, bottom=1in, left=1in, right=1in]{geometry}
\usepackage{mathrsfs}

\usetikzlibrary{shapes}

\newtheorem{theoremx}{Theorem}

\newtheorem{theorem}{Theorem}[section]

\newtheorem{corollary}[theorem]{Corollary}
\newtheorem{lemma}[theorem]{Lemma}
\newtheorem{proposition}[theorem]{Proposition}

\theoremstyle{definition}
\newtheorem{definition}[theorem]{Definition}

\newtheorem{notation}[theorem]{Notation}

\newtheorem{example}[theorem]{Example}

\newtheorem{conjecture}[theorem]{Conjecture}

\newtheorem{remark}[theorem]{Remark}
\numberwithin{equation}{subsection}

\newcommand{\set}[1]{\left\{ #1 \right\}}
\newcommand{\NN}{\mathbb{N}}
\newcommand{\CC}{\mathbb{C}}

\newcommand{\cM}{\mathcal{T}}
\newcommand{\cL}{\mathcal{L}}

\newcommand{\cJ}{\mathcal{J}}
\newcommand{\cT}{\mathcal{T}}

\newcommand{\mA}{\mathfrak{A}}

\newcommand{\pT}{\CC[x_1,\ldots,x_s]/\langle f,\cJ_n(f)\rangle}
\newcommand{\pTdos}{\CC[x_1,x_2]/\langle f,\cJ_n(f)\rangle}

\newcommand{\Der}{\operatorname{Der}}
\newcommand{\Jac}{\operatorname{Jac}}
\newcommand{\conv}{\CC\{x_1,\ldots,x_s\}}

\newcommand{\pol}{\CC[x_1,\ldots,x_s]}
\newcommand{\Lsc}{B}
\newcommand{\Ls}{A}

\newcommand{\Al}{\mathcal{A}}

\newcommand{\Card}{\operatorname{Card}}

\begin{document}
\newcommand{\tens}{\otimes}
\newcommand{\hhtest}[1]{\tau ( #1 )}
\renewcommand{\hom}[3]{\operatorname{Hom}_{#1} ( #2, #3 )}

\allowdisplaybreaks

\title[Higher Jacobian matrix of weighted homogeneous polynomials]{Higher Jacobian matrix of weighted homogeneous polynomials and derivation algebras}

\author[W. Badilla-C\'espedes]{Wágner Badilla-Céspedes}
\address{Centro de Ciencias Matemáticas, UNAM, Campus Morelia, Morelia, Michoacán, México.}
\email{wagner@matmor.unam.mx }

\author[A. Castorena]{Abel Castorena}

\email{abel@matmor.unam.mx}

\author[D. Duarte]{Daniel Duarte}

\email{adduarte@matmor.unam.mx}


\maketitle

\begin{abstract}
We prove that the ideal generated by the maximal minors of the higher-order Jacobian matrix of a weighted homogeneous polynomial is also weighted homogeneous. As an application, we give a partial answer to a conjecture concerning the non-existence of negative weight derivations on the higher Nash blowup local algebra of a hypersurface.
\end{abstract}


\section*{Introduction}

The  Jacobian matrix of order $n$ is a higher-order version of the classical Jacobian matrix. It was introduced as a tool for computing the higher Nash blowup of a hypersurface \cite{Duarte2017}. Higher-order Jacobian matrices were later rediscovered and further developed by several authors \cite{BJNB2019,BD2020,MR4229623}. Ever since, this matrix has seen a wide variety of applications: in the study of higher Nash blowups  \cite{Duarte2017,MR4229623}, the study of invariants of rings in positive characteristic \cite{BJNB2019}, the study of homological properties of the module of higher K\"ahler differentials \cite{BD2020,AD2021}, relations with Hasse-Schmidt derivations and jet schemes \cite{LDS2023, Barajas2023}, invariants of germs of hypersurface singularities \cite{HMYZ2023,LY2023}, motivic zeta functions \cite{LY2023}. In addition to these applications, new questions has been raised regarding general properties of higher Jacobian matrices.

Let $f\in\conv$, where $\conv$ denotes the ring of convergent power series at the origin. Let $\Jac_n(f)$ be the higher Jacobian matrix of $f$ and $\cJ_n(f)$ the ideal generated by the maximal minors of $\Jac_n(f)$ (see Definition \ref{Jacn}). N. Hussain, G. Ma, S. S.-T. Yau, and H. Zuo defined the higher Nash blowup local algebra as the quotient $\cM_n(f)=\conv/\langle f,\cJ_n(f) \rangle$ \cite{HMYZ2023}. For $n=1$, this is the classical Tjurina algebra. The authors proposed several conjectures regarding $\cT_n(f)$: invariance under contact equivalence, homogeneity properties, and bounds for their dimensions. They also provided some evidence that supports the conjectures. These questions can be seen as higher-order analogues of classical results in singularity theory. For instance, a famous result by J. Mather and S. S.-T. Yau states that the Tjurina algebra is a complete invariant of isolated hypersurface singularities under contact equivalence \cite{MR0674404}. Moreover, higher-order versions of  the Tjurina algebra have been an object of intense study (see, for instance, \cite{MR2290112,MR3398724,MR3606998,MR4129535,MR4594780}).

Very recently, the conjecture regarding the invariance of $\cT_n(f)$ under contact equivalence was proved by Q. T. L\^{e} and T. Yasuda \cite{LY2023}. In this paper we are interested in exploring the questions regarding the homogeneity properties of the higher Jacobian matrix. The following is our first main theorem.

\begin{theoremx}[{see Theorem \ref{theorem:Tjurina-graded}}] \label{theorem-main-1}
Let $f\in \CC[x_1,\ldots,x_s]$ be a weighted homogeneous polynomial. Then $\cJ_n(f)$ is a weighted homogeneous ideal for every $n \in \NN$.
\end{theoremx}

Another object introduced by the aforementioned authors is the algebra of derivations $\cL_n(f)=\Der(\cM_n(f))$. For $n=1$, this is known as the Yau algebra of $f$. A classical result in singularity theory states that the Yau algebra is a solvable Lie algebra for isolated hypersurface singularities \cite{Y1986,Y1991}. Moreover, the Yau algebra has been used to distinguish complex analytic structures of isolated hypersurface singularities \cite{MR1032879}. 

By Theorem \ref{theorem-main-1}, $\pT$ is a graded algebra whenever $f$ is a weighted homogeneous polynomial. In this case, the algebra $\Der(\pT)$ is also graded. There is a general conjecture on algebras of derivations of local Artinian graded algebras, known as Halperin conjecture, stating the non-existence of negative weight derivations over those algebras \cite{H1999}. This conjecture has been intensively studied \cite{M1982,FHT2001,CYZ2019,CCYZ2020,CHYZ2020}. Inspired by Halperin's question, N. Hussain, G. Ma, S. S.-T. Yau, and H. Zuo conjectured that there are no negative weight derivations on $\pT$. They proved their conjecture in the case $n=s=2$. Our second main result shows that the conjecture is true for $n\geq3$ and $s=2$.

\begin{theoremx}[{see Theorem \ref{theorem-derivation}}] \label{theorem-main-2}
Let $f\in\CC[x_1,x_2]$ be a weighted homogeneous polynomial of weight $(w_1,w_2)\in\NN^2$ and degree $d$ that defines an isolated hypersurface singularity. Suppose that $d\geq 2w_1\geq 2w_2>0$. Let $n\geq3$. Then $\Der(\pTdos)$ is non-negatively graded.

\end{theoremx}

This paper is divided as follows. In the first section we prove Theorem \ref{theorem-main-1}. We also establish some facts regarding the weighted degree of the maximal minors of the higher Jacobian matrix. The second section contains the proof of Theorem \ref{theorem-main-2}.

\section{Higher-order Jacobian matrix of a weighted homogeneous polynomial}

The following notation will be constantly used throughout this paper. 

\begin{notation}\label{nota}
Let $n,s\in\NN_{\geq1}$. Given $\gamma=(\gamma_1,\ldots,\gamma_s)\in\NN^s$, denote $|\gamma|=\gamma_1+\cdots+\gamma_s$. For $\gamma,\gamma'\in\NN^s$ we denote $\gamma\cdot\gamma'$ the usual Euclidean inner product. In addition, consider
\begin{align*}
&\Lsc:=\{\beta\in\NN^s|0\leq|\beta|\leq n-1\},\\
&\Lsc_1:=\{\beta\in\NN^s|1\leq|\beta|\leq n-1\},\\
&\Ls:=\{\alpha\in\NN^s|0\leq|\alpha|\leq n\},\\
&\Ls_1:=\{\alpha\in\NN^s|1\leq|\alpha|\leq n\},\\
&C_n:=\{\alpha\in\NN^s||\alpha|=n\}.
\end{align*}
\end{notation}

\begin{remark}\label{card}
It is known that
\begin{align*}
&\Card(\Lsc)=\binom{s+n-1}{s}=:M,\\
&\Card(\Ls)=\binom{s+n}{s}=:N,\\
&\Card(C_n)=\binom{n+s-1}{s-1}=N-M=:l.
\end{align*}
\end{remark}

\begin{definition}[{\cite{Duarte2017,BJNB2019,BD2020}}] \label{Jacn}
Let $f\in\conv$. Denote
$$\Jac_n(f):=\left(\frac{1}{(\alpha-\beta)!}\frac{\partial^{\alpha-\beta}(f)}{\partial x^{\alpha-\beta}}\right)_{\substack{\beta\in\Lsc \\ \alpha\in\Ls_1}},$$
where we define $\displaystyle{\frac{1}{(\alpha-\beta)!}\frac{\partial^{\alpha-\beta}(f)}{\partial x^{\alpha-\beta}}=0},$ whenever $\alpha_i<\beta_i$ for some $i$. It is a $M\times (N-1)$-matrix. We call $\Jac_n(f)$ the higher-order Jacobian matrix of $f$ or the Jacobian matrix of order $n$ of $f$. We order this matrix increasingly using a graded lexicographical order and taking $x_1 < x_2 <\cdots< x_s$. Moreover, denote as $\cJ_n(f)$ the ideal generated by all maximal minors of $\Jac_n(f)$.
\end{definition}

\begin{remark}
The higher-order Jacobian matrix was originally introduced as a tool to compute the higher Nash blowup of a hypersurface \cite{MR2371378,Duarte2017}.
\end{remark}

\begin{example}\label{e3-1}
Let $f=x^3-y^2\in\CC[x,y]$. Then 
\[\Jac_2(f)=
\left( 
\begin{array}{ccccc}
             3x^2& -2y & 3x & 0 & -1 \\
             f & 0 & 3x^2 & -2y & 0 \\
             0 & f & 0 & 3x^2 & -2y\\
\end{array} 
\right).\] 
\end{example}

N. Hussain, G. Ma, S. S.-T. Yau, and H. Zuo proposed the following conjecture regarding homogeneity properties of $\Jac_n(f)$.

\begin{conjecture}[{\cite[Conjecture 1.7]{HMYZ2023}}]\label{conjecture:Tjurina-graded}
Let $f\in\pol$ be a weighted homogeneous polynomial. Then $\cJ_n(f)$ is a weighted homogeneous ideal.
\end{conjecture}

The aforementioned authors verified Conjecture \ref{conjecture:Tjurina-graded} in the case $s=n=2$ \cite[Lemma 4.1]{HMYZ2023}. The following theorem gives a positive answer to Conjecture \ref{conjecture:Tjurina-graded} for arbitrary $n$ and $s$.

\begin{theorem}\label{theorem:Tjurina-graded}
Let $f$ be a weighted homogeneous polynomial in $\CC[x_1,\ldots,x_s]$ of weight $w \in \NN^s$ and degree $d$. Then $\cJ_n(f)$ is a weighted homogeneous ideal for every $n \in \NN$.   
\end{theorem}
\begin{proof}
Let $L$ be a matrix formed by taking $M$ columns of $\Jac_n(f)$. Let $B=\{\beta(1),\ldots,\beta(M)\}$ and $\{\alpha(1),\ldots,\alpha(M)\}\subset\Ls_1$ be the sets of vectors indexing the rows and columns of $L$, respectively. Denote the entries of $L$ as $L_{\beta(i)\alpha(j)}$ with $1 \leq i,j \leq M$. Recall that $$L_{\beta(i)\alpha(j)}=\frac{1}{(\alpha(j)-\beta(i))!}\frac{\partial^{\alpha(j)-\beta(i)}(f)}{\partial x^{\alpha(j)-\beta(i)}}.$$ 

Set $I=\set{\theta \in S_{M} \;| \; L_{\beta(i)\alpha(\theta(i))}\not =0 \; \mathrm{for \, every}\; 1 \leq i \leq M}$, where $S_{M}$ denotes the symmetric group of $M$ elements. We have that 
$$\det(L)=\sum_{\theta \in I}\operatorname{sgn}(\theta) L_{\beta(1)\alpha(\theta(1))}\cdots L_{\beta(M)\alpha(\theta(M))}.$$ 
Notice that $f$ being weighted homogeneous of degree $d$ implies that $L_{\beta(i)\alpha(\theta(i))}$ is a weighted homogeneous polynomial of degree $d-(\alpha(\theta(i))-\beta(i))\cdot w$. Hence, $L_{\beta(1)\alpha(\theta(1))}\cdots L_{\beta(M)\alpha(\theta(M))}$ is a weighted homogeneous polynomial of degree $dM-\sum_{i=1}^{M}(\alpha(\theta(i))-\beta(i))\cdot w$.

For every $\sigma \in I$ we claim that 
$$\sum_{i=1}^{M}(\alpha(\sigma(i))-\beta(i))\cdot w=\sum_{i=1}^{M}(\alpha(i)-\beta(i))\cdot w.$$  
Indeed, 
\begin{align*}
    0&=\sum_{i=1}^{M}(\alpha(\sigma(i))-\alpha(i))\cdot w \\
    &=\sum_{i=1}^{M}(\alpha(\sigma(i))-\beta(i) +\beta(i) -\alpha(i))\cdot w \\
    &=\sum_{i=1}^{M}(\alpha(\sigma(i))-\beta(i))\cdot w-\sum_{i=1}^{M}(\alpha(i)-\beta(i))\cdot w.
\end{align*}
Taking $c= \sum_{i=1}^{M}(\alpha(i)-\beta(i))\cdot w$, we conclude $L_{\beta(1)\alpha(\theta(1))}\cdots L_{\beta(M)\alpha(\theta(M))}$ is a weighted homogeneous polynomial of weight $w$ and degree $dM-c$ for every $\theta \in I$. Therefore, $\det(L)$ is a weighted homogeneous polynomial of weight $w$ and degree $dM-c$.
\end{proof}

Notice that the proof of the previous theorem also exhibited the degree of the maximal minors of $\Jac_n(f)$. Our next goal is to give lower bounds for those degrees, in some special cases.

\subsection{Homogeneous case}

Let $f$ be a homogeneous polynomial in $\CC[x_1,\ldots,x_s]$ of degree $d$, and let $n\geq 2$, $s\geq2$. Let $g$ be the determinant of the submatrix of $\Jac_n(f)$ formed by the  columns indexed by $\{\alpha(1),\ldots,\alpha(M)\}\subset A_1$ and the rows indexed by $B=\{\beta(1),\ldots,\beta(M)\}$. From the proof of Theorem \ref{theorem:Tjurina-graded}, $g$ is a homogeneous polynomial of degree $dM-c$, where
\begin{align*}
c&=\sum_{i=1}^{M}|\alpha(i)-\beta(i)| =\sum_{i=1}^{M}|\alpha(i)|-\sum_{i=1}^M |\beta(i)|.
\end{align*}

\begin{lemma}\label{lemma:degree-homogeneous} 
 With the previous notation we have $sM\geq c+s$.
 \end{lemma}
\begin{proof}
We can assume that $1\leq|\alpha(1)|\leq\cdots\leq|\alpha(M)|\leq n$. Recall the notation from Remark \ref{card} and denote $\mathcal{A}:=\{\alpha(1),\ldots,\alpha(M-l)\}$. Notice that $\Al\subset B_1$, thus
\begin{align*}
\sum_{i=1}^M|\alpha(i)| &\leq \sum_{\alpha\in\mathcal{A}}|\alpha|+\sum_{\alpha \in C_n}|\alpha|=\sum_{\alpha\in\mathcal{A}}|\alpha|+ln=\sum_{\alpha\in\mathcal{A}}|\alpha|+sM.
\end{align*}  

On the other hand, $\Card(B_1\setminus\Al)=(M-1)-(M-l)=l-1\geq s$. This implies
$$\sum_{i=1}^M|\beta(i)|=\sum_{\beta\in B_1}|\beta| \geq \sum_{\alpha\in\Al}|\alpha|+s.$$

From the above, 
 \begin{align*}
c+s &= \sum_{i=1}^M|\alpha(i)|-\sum_{i=1}^{M}|\beta(i)| + s\\
&\leq sM+\sum_{\alpha \in \mathcal{A}}|\alpha|-\sum_{i=1}^M|\beta(i)|+s\\
&\leq sM.
 \end{align*}
\end{proof}

\begin{corollary}\label{coro homog}
Let $f\in\CC[x_1,\ldots,x_s]$ be a homogeneous polynomial of degree $d\geq s$ and let $n\geq 2$. Then $d$ is a lower bound for the degrees of the maximal minors of $\Jac_n(f)$. In other words, $d\geq s$ implies $dM-c\geq d$.
\end{corollary}
\begin{proof}
By Lemma \ref{lemma:degree-homogeneous}, $d(M-1)-c\geq s(M-1)-c\geq 0$.
\end{proof}

\subsection{Weighted homogeneous case}

In this subsection we work with weighted homogeneous polynomials assuming $s=2$ and $n \geq 3$. Recall Notation \ref{nota}. In this case we have $C_n=\{(n-i,i) \;|\; 0 \leq i \leq n\}$. Notice that $\Card(C_n)=n+1$ and $M=\frac{n(n+1)}{2}$.

As in the previous subsection, we want to give a lower bound for the weighted degrees of the maximal minors of $\Jac_n(f)$. Because we are working with an arbitrary weight, the strategy for finding such bound is not as straightforward as the ones of Lemma \ref{lemma:degree-homogeneous} and Corollary \ref{coro homog}.

\begin{lemma}\label{lemma:combinatorics}
Let $w=(w_1,w_2) \in \NN^2$. Then $\displaystyle\sum_{\alpha \in C_n}\alpha \cdot w =|w|M$.
\end{lemma}
\begin{proof}
We have that
\begin{align*}
\sum_{\alpha \in C_n}\alpha \cdot w&= \sum_{i=0}^{n} (n-i,i) \cdot (w_1,w_2)\\ 
&=\sum_{i=0}^{n}(n-i)w_1+ \sum_{i=0}^{n}i w_2\\
&=\frac{n(n+1)}{2}w_1 + \frac{n(n+1)}{2}w_2 \\ 
&=|w|M.
\end{align*}
\end{proof}

Before going any further, we present an example that illustrates the notation and the key idea we develop to find the desired bound.

\begin{example}
Let $n=3$. Thus $M=6$. In this case,
\begin{align*}
B&=\{(0,0),(1,0),(0,1),(2,0),(1,1),(0,2)\},\\
B_1&=\{(1,0),(0,1),(2,0),(1,1),(0,2)\},\\
A_1&=\{(1,0),(0,1),(2,0),(1,1),(0,2),(3,0),(2,1),(1,2),(0,3)\},\\
C_3&=\{(3,0),(2,1),(1,2),(0,3)\}.
\end{align*}
Consider the submatrix of $\Jac_3(f)$ defined by the columns 
$$\mA=\{(1,0),(0,1),(2,0),(1,1),(3,0),(1,2)\}.$$ 
Let us show how we bound the sums $\sum_{\alpha\in\mA}\alpha\cdot w$ and $\sum_{\beta\in B}\beta\cdot w$.

Let $C=\mA\setminus C_3=\{(1,0),(0,1),(2,0),(1,1)\}$. Let $\theta=(1,0)$ and $\tau =(2,1)$. Notice that $\theta\in C$, $\tau\in C_3\setminus\mA$, and $\theta \cdot w < \tau \cdot w$. This implies

\begin{align*}
\sum_{\alpha\in\mA}\alpha\cdot w&=\sum_{\alpha\in C}\alpha\cdot w+\sum_{\alpha\in \mA\cap C_3}\alpha\cdot w \notag\\
&=\sum_{\alpha\in C\setminus\{\theta\}}\alpha\cdot w+\theta\cdot w+\sum_{\alpha\in \mA\cap C_3}\alpha\cdot w \notag\\
&<\sum_{\alpha\in C\setminus\{\theta\}}\alpha\cdot w+\tau\cdot w+\sum_{\alpha\in \mA\cap C_3}\alpha\cdot w \notag\\
&<\sum_{\alpha\in C\setminus\{\theta\}}\alpha\cdot w+\sum_{\alpha\in C_3}\alpha\cdot w\notag\\
&=\sum_{\alpha\in C\setminus\{\theta\}}\alpha\cdot w+|w|M.
\end{align*}
    
On the other hand, we also have
\begin{align*}
\sum_{\beta\in B}\beta\cdot w=\sum_{\beta\in B_1}\beta\cdot w&=\Big(\sum_{\substack{\beta\in B_1\\\beta\neq (1,0),(0,2)}}
\beta\cdot w\Big)+(1,0)\cdot w+(0,2)\cdot w\\    
&=\Big(\sum_{\alpha\in C\setminus\{\theta\}}\alpha\cdot w\Big)+(1,0)\cdot w+(0,2)\cdot w\\
&>\sum_{\alpha\in C\setminus\{\theta\}}\alpha\cdot w+w_1+w_2.
\end{align*}
\end{example}

The following lemma generalizes the process of the previous example.

\begin{lemma} \label{lemma:general-sets}
Let $w=(w_1,w_2) \in \NN^2$. Assume $w_1\geq w_2>0$. Let $\mathfrak{A}=\{\alpha(1),\ldots,\alpha(M)\} \subset A_1$, $B=\{\beta(1),\ldots,\beta(M)\}$, 
 and $C=\mathfrak{A} \setminus C_n$. Define subsets $\mathcal{A}\subset C$ satisfying the following conditions:
\begin{enumerate}
\item $\mathcal{A}=C$, $\Card(B_1 \setminus \mathcal{A})\geq 2$ and there exists $\theta\in B_1\setminus \mathcal{A}$ such that $\theta_1\geq1$.
\item $\mathcal{A}=C\setminus\{\theta\}$,  where $\theta\in C$ is such that $\theta_1\geq1$, $\Card(B_1 \setminus \mathcal{A})\geq 2$, and there exists $\tau\in C_n\setminus \mathfrak{A}$ such that $\theta\cdot w\leq \tau\cdot w$.
\item $\mathcal{A}=C\setminus\{\theta,\theta'\}$, where $\theta,\theta'\in C$ are such that $\theta_1\geq1$ and there exist $\tau,\tau'\in C_n\setminus \mathfrak{A} $ such that $\theta\cdot w\leq \tau\cdot w$ and $\theta'\cdot w\leq \tau'\cdot w$.
\end{enumerate}
Then $|w|M\geq c+|w|$, where $c =\sum_{i=1}^{M}(\alpha(i)-\beta(i))\cdot w$. 
\end{lemma}
\begin{proof}
In each of the cases we have:
\begin{align*}
\sum_{i=1}^M\alpha(i)\cdot w&=\sum_{\alpha\in C}\alpha\cdot w+\sum_{\alpha\in \mathfrak{A}\cap C_n}\alpha\cdot w \notag\\
&=\sum_{\alpha\in\mathcal{A}}\alpha\cdot w+\sum_{\alpha\in C\setminus\mathcal{A}}\alpha\cdot w+\sum_{\alpha\in \mathfrak{A}\cap C_n}\alpha\cdot w \notag\\
&\leq \sum_{\alpha\in\mathcal{A}}\alpha\cdot w+\sum_{\alpha\in C_n}\alpha\cdot w\\
&=\sum_{\alpha\in\mathcal{A}}\alpha\cdot w+|w|M,
\end{align*}
where the last equality follows by Lemma \ref{lemma:combinatorics}.

On the other hand, assuming that $\beta(1)=(0,0)$, since $\Card(B_1\setminus \mathcal{A})\geq2$ we also have:
$$\sum_{i=1}^M\beta(i)\cdot w=\sum_{i=2}^M\beta(i)\cdot w\geq\sum_{\alpha\in\mathcal{A}}\alpha\cdot w+(w_1+w_2).$$
Finally,
\begin{align*}
c&=\sum_{i=1}^M\alpha(i)\cdot w-\sum_{i=1}^M\beta(i)\cdot w\notag\\
&\leq |w|M+\sum_{\alpha\in\mathcal{A}}\alpha\cdot w-\sum_{i=1}^M\beta(i)\cdot w\notag\\
&\leq |w|M+\sum_{\alpha\in\mathcal{A}}\alpha\cdot w-\sum_{\alpha\in\mathcal{A}}\alpha\cdot w-(w_1+w_2)\\
&=|w|M-|w|.\notag
\end{align*}
\end{proof}

\begin{remark}\label{remark:degree}
 Let $f$ be a weighted homogeneous polynomial in $\CC[x_1,x_2]$ of weight $w \in \NN^2$ and degree $d$. Without loss of generality we can assume that $w_1\geq w_2>0$. Let $n \geq 3$ and
 $g$ be the determinant of the submatrix of $\Jac_n(f)$ formed by the  columns indexed by $ \mathfrak{A}=\{\alpha(1),\ldots,\alpha(M)\}\subset A_1$ and the rows indexed by $B=\{\beta(1),\ldots,\beta(M)\}$. From Theorem \ref{theorem:Tjurina-graded}, $g$ is a weighted homogeneous polynomial of weight $w$ and degree $dM-c$, where
 \begin{align*}
c &=\sum_{i=1}^{M}(\alpha(i)-\beta(i))\cdot w \\
&=\sum_{i=1}^{M}\alpha(i)\cdot w-\sum_{i=1}^{M}\beta(i)\cdot w.
 \end{align*}
 We take $C=\mathfrak{A}\setminus C_n \subset B_1$. We note that $\mathfrak{A}\cap C_n \not= \emptyset$ since $\Card(B_1)=M-1$. Hence $M-(n+1) \leq \Card(C)\leq M-1$. Furthermore, denote $I_j=\{(0,1),(0,2),\ldots,(0,j)\}$ for $1 \leq j \leq n$, and $C_i=\{\beta\in\NN^2||\beta|=i\}$, for $0\leq i\leq n$.
 \end{remark}

\begin{proposition}\label{lemma:special-case}
 Consider the assumptions of Remark \ref{remark:degree}.  If $\Card(C)=M-n$, then the following cases hold.
 \begin{enumerate}
 \item If $\mathfrak{A}=(B_1 \cup C_n) \setminus I_n$, then $c=w_1M$.
 \item If $\mathfrak{A}\not=(B_1 \cup C_n) \setminus I_n$, then $|w|M\geq c +|w|$. 
 \end{enumerate}
 \end{proposition}
 \begin{proof}
In order to show (1), we first notice
\begin{align*}
\sum_{i=1}^{M}\beta(i)&= \sum_{i=0}^{n-1} 
\left( \sum_{\beta \in C_i}\beta \right) \\
&=\sum_{i=0}^{n-1}\left(\frac{i(i+1)}{2},\frac{i(i+1)}{2}\right)\\
&=\sum_{i=1}^{n}\left(\frac{(i-1)i}{2},\frac{(i-1)i}{2}\right).
\end{align*}

Besides, we note that
\begin{align*}
\sum_{i=1}^{M}\alpha(i)&= \sum_{i=1}^{n} 
\left( \sum_{\substack{\alpha \in C_i \\ \alpha \not= (0,i)}}\alpha \right) \\
&=\sum_{i=1}^{n}\left(\frac{i(i+1)}{2},\frac{i(i-1)}{2}\right).
\end{align*}

Therefore, $c =\sum_{i=1}^{M}(\alpha(i)-\beta(i))\cdot w=\sum_{i=1}^{n}(i,0)\cdot w=w_1M$.

For (2) we consider two cases:

\begin{itemize}
\item $C \not=B_1\setminus I_{n-1}$. In this case, there exists $\theta \in B_1 \setminus C$ such that $\theta_1\geq 1$. Set $\mathcal{A}=C$ and so $\Card(B_1\setminus\mathcal{A})\geq2$. Then $|w|M \geq c+|w|$ by Lemma \ref{lemma:general-sets} (1). 
\item $C =B_1\setminus I_{n-1}$. In this case $\theta=(1,0)\in C$, and $C_n \setminus \mathfrak{A}$ has a unique element $\tau\not=(0,n)$. Then, $\theta \cdot w \leq \tau \cdot w$. Set $\mathcal{A}=C\setminus \{\theta\}$ and so $\Card(B_1\setminus\mathcal{A})\geq2$. Then $|w|M \geq c+|w|$ by Lemma \ref{lemma:general-sets} (2).
\end{itemize}
 \end{proof}

\begin{proposition}\label{lemma:degree}
Consider the assumptions of Remark \ref{remark:degree}. If $\Card(C)\not=M-n$, then $|w|M \geq c+|w|$. 
\end{proposition}
\begin{proof} 
We proceed by analyzing all possible cases we can have. Assume $B_1=B\setminus\{\beta(1)\}$.

\begin{enumerate}
\item $\text{Card}(C)=M-1$. In this case we have $C=B_1$. We have that $\theta=(1,0),\theta'=(0,1)\in C$. Since $\Card(C_n\setminus\mathfrak{A})=n\geq3$, there exist two vectors $\tau,\tau'\in C_n\setminus\mathfrak{A}$, 
$\tau\neq\tau'$ such that $\tau_1,\tau'_2 \geq 1$, so that $\theta\cdot w\leq \tau \cdot w$ and $\theta'\cdot w\leq\tau' \cdot w$. In this case we take $\mathcal{A}=C\setminus\{\theta,\theta'\}$ and the result follows by Lemma \ref{lemma:general-sets}(3).

\item $\text{Card}(C)=M-2$. In this case $C=\mathfrak{A}\setminus C_n=B_1\setminus\{\beta(j)\}$ for some $j\geq2$. We have two subcases:

\begin{enumerate}
\item $\theta=(1,0) \in C$. In this case, since $\Card(C_n\setminus\mathfrak{A})=n-1\geq 2$, there exists $\tau \in C_n\setminus \mathfrak{A}$, such that $\tau_1 \geq 1$ and $\theta \cdot w \leq \tau\cdot w$. We can take $\mathcal{A}=C\setminus\{\theta\}$ and so $\Card(B_1\setminus\mathcal{A})\geq2$. The result follows by Lemma \ref{lemma:general-sets}(2).
\item $(1,0) \not \in C$. We have that $\{(2,0),(1,1)\} \subset C$. As $\Card(C_n\setminus\mathfrak{A})=n-1\geq 2$, there exists $\tau\in C_n\setminus\mathfrak{A}$ such that $\tau_1,\tau_2 \geq 1$ or $\tau=(n,0)$. If $\tau_1,\tau_2 \geq 1$ let $\theta=(1,1)$. If $\tau=(n,0)$ let $\theta=(2,0)$. In any of the cases $\theta  \in C$ and $\theta\cdot w\leq\tau\cdot w$. We can take $\mathcal{A}=C\setminus\{\theta\}$ and so $\Card(B_1\setminus\mathcal{A})\geq2$. The result follows by Lemma \ref{lemma:general-sets}(2). 
\end{enumerate}

\item$\text{Card}(C)=M-j$ for $3\leq j\leq n-1$. We have two subcases:

\begin{enumerate}
\item$C=B_1\setminus\{(0,b_1),...,(0,b_{j-1})\}$. Then $\theta=(1,0)\in C$ and $\Card(C_n\setminus\mathfrak{A})=n+1-j\geq 2$. Then there exists $\tau \in C_n\setminus\mathfrak{A}$ such that $\theta \cdot w\leq\tau \cdot w$. We can take $\mathcal{A}=C\setminus\{\theta\}$ and so $\Card(B_1\setminus\mathcal{A})\geq2$. The result follows by Lemma \ref{lemma:general-sets}(2). 
\item $C=B_1\setminus\{(a_1,b_1),...,(a_{j-1},b_{j-1})\}$ with at least one subindex $i$, $1\leq i\leq j-1$ such that $\theta=(a_i,b_i)$ satisfies that $a_i\geq 1$. In this case, $\theta \in B_1 \setminus C$. We can take $\mathcal{A}=C$ and so $\Card(B_1\setminus\mathcal{A})\geq2$. The result follows by Lemma \ref{lemma:general-sets}(1).
\end{enumerate}

\item $\Card(C)=M-(n+1)$. In this situation there exists $\theta \in B_1\setminus C$ with $\theta_1 \geq 1$. In this case we can take $\mathcal {A}=C$ and so $\Card(B_1\setminus\mathcal{A})\geq2$. The result follows by Lemma \ref{lemma:general-sets}(1).

\end{enumerate}

\end{proof}

\begin{corollary}\label{coro:s=2}
Let $f\in\CC[x_1,x_2]$ be a weighted homogeneous polynomial of weight $w\in\NN^2$ and degree $d$. Suppose that $d\geq 2w_1\geq 2w_2>0$. Let $n\geq 3$ and $g$ be a maximal minor of $\Jac_n(f)$. Then $g$ is a weighted homogeneous polynomial of degree greater or equal than $d$.
\end{corollary}
\begin{proof}
By Theorem \ref{theorem:Tjurina-graded}, $g$ is a weighted homogeneous polynomial of degree $dM-c$. 
Propositions \ref{lemma:special-case} and \ref{lemma:degree} imply $d(M-1)\geq 2w_1(M-1)\geq c.$
\end{proof}

\section{Derivations of the higher Nash blowup local algebra}

In this section we study the Lie algebra of derivations of the higher Nash blowup local algebra of a hypersurface defining an isolated singularity.

\begin{definition}\label{Tjurina}\cite[Definition 1.3]{HMYZ2023}.
Let $f\in\conv$. Denote
$$\cT_n(f)=\conv/\langle f,\cJ_n(f) \rangle.$$
$\cT_n(f)$ is called the higher Nash blowup local algebra of $f$. For $n=1$ this is also known as the Tjurina algebra of $f$. Moreover, denote as $\cL_n(f)=\Der(\cT_n(f))$, i.e., the Lie algebra of derivations of $\cT_n(f)$.
\end{definition}

\begin{remark}
The higher Nash blowup local algebra is denoted as $\mathcal{M}_n$ in \cite{HMYZ2023}. In this paper we change the notation in order to be consistent with the usual notation for the Tjurina algebra. In addition, it is also customary to use $\mathcal{M}_1$ (or, more precisely, $\mathcal{M}$) for the Milnor algebra.
\end{remark}

Let $f$ be a weighted homogeneous polynomial defining an isolated hypersurface singularity. Theorem \ref{theorem:Tjurina-graded} implies that $\pT$ is graded. There is an induced grading on $\Der(\pT)$, as we explain next.

\begin{remark}
Suppose that the hypersurface defined by $f$ has an isolated singularity. Then $\pT$ is Artinian. This follows from the fact that the zero locus of $\langle f,\cJ_n(f)\rangle$ coincides with the singular locus of the hypersurface \cite[Corollary 2.2]{Duarte2017}.
\end{remark}

\begin{lemma}[{\cite[Lemma 2.1]{XY1996}}] \label{der graded}
Let $A=\bigoplus_{i=0}^tA_i$ be a graded commutative Artinian local algebra. Let $L(A)$ be the derivation algebra of $A$. Then $L(A)$ can be graded as follows: $L(A)=\bigoplus_{k=-t}^tL_k$, where $L_k=\{D\in L(A)|D(A_i)\subset A_{i+k} \mbox{ for all }i\}$.
\end{lemma}

N. Hussain, G. Ma, S. S.-T. Yau, and H. Zuo proposed the following conjecture regarding the non-existence of negative weight derivations on $\mathcal{L}_n(f)$.

\begin{conjecture}[{\cite[Conjecture 1.7]{HMYZ2023}}]\label{conjecture:derivation}
Let $f\in\pol$ be a weighted homogeneous polynomial of weight $w\in\NN^s$ and degree $d$ that defines an isolated hypersurface singularity. Suppose that $d\geq 2w_1\geq \cdots\geq 2w_s>0$. Then there is no non-zero negative weight derivation on $\cT_n(f)$, i.e., $\mathcal{L}_n(f)$ is non-negatively graded.  
\end{conjecture}

The same authors verified Conjecture \ref{conjecture:derivation} in the case $s=n=2$ \cite[Theorem B]{HMYZ2023}. The following theorem gives a positive answer to Conjecture \ref{conjecture:derivation} for $s=2$ and $n\geq3$. Together, these results settle the conjecture for two variables.

\begin{theorem}\label{theorem-derivation}
Let $f\in\CC[x_1,x_2]$ be a weighted homogeneous polynomial of weight $(w_1,w_2)\in\NN^2$ and degree $d$ that defines an isolated hypersurface singularity. Suppose that $d\geq 2w_1\geq 2w_2>0$. Let $n\geq3$. Then $\Der(\pTdos)$ is non-negatively graded.
\end{theorem}
\begin{proof}
Let $\overline{D}$ be a derivation of negative degree $-k$ of $\pTdos$. Then $\overline{D}$ corresponds to a homogeneous derivation $D$ of $\CC[x_1,x_2]$ such that $D(\langle f,\cJ_n(f)\rangle)\subset \langle f,\cJ_n(f)\rangle$ \cite[Theorem 2.2]{YZ2016}. Since $D$ is homogeneous, there are weighted homogeneous polynomials $h_1,h_2$ of degree $d_i=-k+w_i$, $i=1,2$, such that $D=h_1\partial_{x_1}+h_2\partial_{x_2}$. It is known that $D(\langle x_1,x_2 \rangle)\subset \langle x_1,x_2 \rangle$ \cite[Lemma 2.5]{XY1996}. This fact, together with $D$ being of negative degree and $w_1\geq w_2$, implies that $h_2=0$ and $h_1=cx_2^b$, for some $c\in\CC$ and $b\geq1$. Hence, $D=cx_2^b\partial_{x_1}$.

By the assumptions on the weight and degree of $f$, and the fact that $f$ defines an isolated singularity, we have that $f$ must be one of the following cases 
\cite[Lemma 2.1]{CXY1995}:
\begin{itemize}
\item[(1)] $f=x_1^m+a_1(x_2)x_1^{m-1}+\cdots+a_m(x_2)$.
\item[(2)] $f=x_1^mx_2+a_1(x_2)x_1^{m-1}+\cdots+a_m(x_2)$.
\end{itemize}


By Corollary \ref{coro:s=2}, $f$ is a generator of smallest degree among the generators of $\langle f,\cJ_n(f)\rangle$. Since $D(f)\in\langle f,\cJ_n(f)\rangle$ and $D$ is of negative degree, it follows that $D(f)=0$, i.e., $cx_2^b\partial_{x_1}(f)=0$. Using (1) and (2) we obtain $\partial_{x_1}(f)\neq0$. We conclude that $c=0$ and so $D=0$.
\end{proof}

\begin{remark}
Consider the notation of Theorem \ref{theorem-derivation}. If $f\in\langle x_1,x_2\rangle^3$, then we can assume $d\geq 2w_1\geq 2w_2>0$ without loss of generality \cite[Theorem 2.1]{MR3490075}, \cite{Saito1971}.
\end{remark}

\section*{Acknowledgements}
The first author was supported by UNAM Posdoctoral Program (POSDOC). The second author was supported with project IN100723, ``Curvas, Sistemas lineales en superficies proyectivas y fibrados vectoriales? from DGAPA, UNAM. The third author is  supported by CONAHCYT project CF-2023-G-33.



\begin{thebibliography}{CMDGF21}

\bibitem[Bar23]{Barajas2023}
Paul Barajas.
\newblock High-order derivations of the hasse-schmidt algebra.
\newblock {\em arXiv:2305.04537}, 2023.

\bibitem[BD20]{BD2020}
Paul Barajas and Daniel Duarte.
\newblock On the module of differentials of order {$n$} of hypersurfaces.
\newblock {\em J. Pure Appl. Algebra}, 224(2):536--550, 2020.

\bibitem[BJNB19]{BJNB2019}
Holger Brenner, Jack Jeffries, and Luis N\'{u}{\~n}ez-Betancourt.
\newblock Quantifying singularities with differential operators.
\newblock {\em Adv. Math.}, 358:106843, 89, 2019.

\bibitem[CCYZ20]{CCYZ2020}
Bingyi Chen, Hao Chen, Stephen S.-T. Yau, and Huaiqing Zuo.
\newblock The nonexistence of negative weight derivations on positive
  dimensional isolated singularities: generalized {W}ahl conjecture.
\newblock {\em J. Differential Geom.}, 115(2):195--224, 2020.

\bibitem[Che99]{H1999}
Hao Chen.
\newblock Nonexistence of negative weight derivations on graded {A}rtin
  algebras: a conjecture of {H}alperin.
\newblock {\em J. Algebra}, 216(1):1--12, 1999.

\bibitem[CHYZ20]{CHYZ2020}
Bingyi Chen, Naveed Hussain, Stephen S.-T. Yau, and Huaiqing Zuo.
\newblock Variation of complex structures and variation of {L}ie algebras {II}:
  {N}ew {L}ie algebras arising from singularities.
\newblock {\em J. Differential Geom.}, 115(3):437--473, 2020.

\bibitem[CMDGF21]{MR4229623}
Enrique Ch\'{a}vez-Mart\'{\i}nez, Daniel Duarte, and Arturo Giles~Flores.
\newblock A higher-order tangent map and a conjecture on the higher {N}ash
  blowup of curves.
\newblock {\em Math. Z.}, 297(3-4):1767--1791, 2021.

\bibitem[CXY95]{CXY1995}
Hao Chen, Yi-Jing Xu, and Stephen S.-T. Yau.
\newblock Nonexistence of negative weight derivation of moduli algebras of
  weighted homogeneous singularities.
\newblock {\em J. Algebra}, 172(2):243--254, 1995.

\bibitem[CYZ19]{CYZ2019}
Hao Chen, Stephen S.-T. Yau, and Huaiqing Zuo.
\newblock Non-existence of negative weight derivations on positively graded
  {A}rtinian algebras.
\newblock {\em Trans. Amer. Math. Soc.}, 372(4):2493--2535, 2019.

\bibitem[dAD21]{AD2021}
Hern\'{a}n de~Alba and Daniel Duarte.
\newblock On the {$k$}-torsion of the module of differentials of order {$n$} of
  hypersurfaces.
\newblock {\em J. Pure Appl. Algebra}, 225(8):Paper No. 106646, 7, 2021.

\bibitem[DGI20]{MR4129535}
Alexandru Dimca, Rodrigo Gondim, and Giovanna Ilardi.
\newblock Higher order {J}acobians, {H}essians and {M}ilnor algebras.
\newblock {\em Collect. Math.}, 71(3):407--425, 2020.

\bibitem[DS15]{MR3398724}
Alexandru Dimca and Gabriel Sticlaru.
\newblock Hessian ideals of a homogeneous polynomial and generalized {T}jurina
  algebras.
\newblock {\em Doc. Math.}, 20:689--705, 2015.

\bibitem[Dua17]{Duarte2017}
Daniel Duarte.
\newblock Computational aspects of the higher {N}ash blowup of hypersurfaces.
\newblock {\em J. Algebra}, 477:211--230, 2017.

\bibitem[FHT01]{FHT2001}
Yves F\'{e}lix, Stephen Halperin, and Jean-Claude Thomas.
\newblock {\em Rational homotopy theory}, volume 205 of {\em Graduate Texts in
  Mathematics}.
\newblock Springer-Verlag, New York, 2001.

\bibitem[GLS07]{MR2290112}
G.-M. Greuel, C.~Lossen, and E.~Shustin.
\newblock {\em Introduction to singularities and deformations}.
\newblock Springer Monographs in Mathematics. Springer, Berlin, 2007.

\bibitem[GP17]{MR3606998}
Gert-Martin Greuel and Thuy~Huong Pham.
\newblock Mather-{Y}au theorem in positive characteristic.
\newblock {\em J. Algebraic Geom.}, 26(2):347--355, 2017.

\bibitem[HMYZ23]{HMYZ2023}
Naveed Hussain, Guorui Ma, Stephen S.-T. Yau, and Huaiqing Zuo.
\newblock Higher {N}ash blow-up local algebras of singularities and its
  derivation {L}ie algebras.
\newblock {\em J. Algebra}, 618:165--194, 2023.

\bibitem[LDS24]{LDS2023}
Yann Le~Dr\'eau and Julien Sebag.
\newblock Arc scheme and higher differential forms.
\newblock {\em Osaka J. Math.}, 61(3):381--390, 2024.

\bibitem[LY25]{LY2023}
Quy~Thuong L\^e and Takehiko Yasuda.
\newblock Higher {J}acobian ideals, contact equivalence and motivic zeta
  functions.
\newblock {\em Rev. Mat. Iberoam.}, 41(3):1057--1080, 2025.

\bibitem[Mei82]{M1982}
W.~Meier.
\newblock Rational universal fibrations and flag manifolds.
\newblock {\em Math. Ann.}, 258(3):329--340, 1981/82.

\bibitem[MY82]{MR0674404}
John~N. Mather and Stephen S.~T. Yau.
\newblock Classification of isolated hypersurface singularities by their moduli
  algebras.
\newblock {\em Invent. Math.}, 69(2):243--251, 1982.

\bibitem[MYZ23]{MR4594780}
Guorui Ma, Stephen S.-T. Yau, and Huaiqing Zuo.
\newblock Nonexistence of negative weight derivations of the local {$k$}-th
  {H}essian algebras associated to isolated singularities.
\newblock {\em Pacific J. Math.}, 323(1):129--172, 2023.

\bibitem[Sai71]{Saito1971}
Kyoji Saito.
\newblock Quasihomogene isolierte {S}ingularit\"{a}ten von {H}yperfl\"{a}chen.
\newblock {\em Invent. Math.}, 14:123--142, 1971.

\bibitem[SY90]{MR1032879}
Craig Seeley and Stephen S.-T. Yau.
\newblock Variation of complex structures and variation of {L}ie algebras.
\newblock {\em Invent. Math.}, 99(3):545--565, 1990.

\bibitem[XY96]{XY1996}
Yi-Jing Xu and Stephen S.-T. Yau.
\newblock Micro-local characterization of quasi-homogeneous singularities.
\newblock {\em Amer. J. Math.}, 118(2):389--399, 1996.

\bibitem[Yas07]{MR2371378}
Takehiko Yasuda.
\newblock Higher {N}ash blowups.
\newblock {\em Compos. Math.}, 143(6):1493--1510, 2007.

\bibitem[Yau86]{Y1986}
Stephen S.-T. Yau.
\newblock Solvable {L}ie algebras and generalized {C}artan matrices arising
  from isolated singularities.
\newblock {\em Math. Z.}, 191(4):489--506, 1986.

\bibitem[Yau91]{Y1991}
Stephen S.-T. Yau.
\newblock Solvability of {L}ie algebras arising from isolated singularities and
  nonisolatedness of singularities defined by {${\rm sl}(2,{\bf C})$} invariant
  polynomials.
\newblock {\em Amer. J. Math.}, 113(5):773--778, 1991.

\bibitem[YZ16a]{YZ2016}
Stephen S.-T. Yau and Huai~Qing Zuo.
\newblock A sharp upper estimate conjecture for the {Y}au number of a weighted
  homogeneous isolated hypersurface singularity.
\newblock {\em Pure Appl. Math. Q.}, 12(1):165--181, 2016.

\bibitem[YZ16b]{MR3490075}
Stephen S.-T. Yau and Huaiqing Zuo.
\newblock Derivations of the moduli algebras of weighted homogeneous
  hypersurface singularities.
\newblock {\em J. Algebra}, 457:18--25, 2016.

\end{thebibliography}

\end{document}